\documentclass[12pt]{elsarticle}
\usepackage{amsmath,amsthm,amscd,amssymb,extsizes}
\usepackage[active]{srcltx}
\newtheorem{theorem}{Theorem}

\newtheorem{corollary}[theorem]{Corollary}

\theoremstyle{remark}

\DeclareMathOperator{\trace}{trace}
\DeclareMathOperator{\diag}{diag}
\DeclareMathOperator{\rank}{rank}
\newcommand{\C}{\mathbb C}
\newcommand{\R}{\mathbb R}
\newcommand{\F}{\mathbb F}

\renewcommand{\ge}{\geqslant}

\begin{document}
\title{Each $n$-by-$n$ matrix with $n>1$ is a sum of 5 coninvolutory matrices}

\author[aba]{Ma. Nerissa M. Abara}
\ead{issa@up.edu.ph}
\address[aba]{Institute of Mathematics, University of the Philippines, Diliman, Quezon City 1101,
Philippines.}

\author[mer]{Dennis I. Merino\corref{cor}}
\ead{dmerino@selu.edu}
\address[aba]{Department of Mathematics, Southeastern Louisiana University, Hammond, LA
70402-0687, USA.}

\author[ser]{Vyacheslav I. Rabanovich}
\ead{rvislavik@gmail.com}

    \author[ser]{Vladimir V.
    Sergeichuk}
\ead{sergeich@imath.kiev.ua}
\address[ser]{Institute of Mathematics,
Tereshchenkivska 3,
Kiev, Ukraine}

\author[aba]{John Patrick Sta. Maria}
\ead{jpbstamaria@gmail.com}
\cortext[cor]{Corresponding author}

\begin{abstract}
An $n\times n$ complex matrix $A$ is called \emph{coninvolutory} if $\bar AA=I_n$ and \emph{skew-coninvolutory} if $\bar AA=-I_n$ (which implies that $n$ is even). We prove that each matrix of size $n\times n$ with $n>1$  is a sum of 5 coninvolutory matrices and each matrix of size $2m\times 2m$ is a sum of 5 skew-coninvolutory matrices.

We also prove that each square complex matrix is a sum of a coninvolutory matrix and  a condiagonalizable matrix. A matrix $M$ is called \emph{condiagonalizable} if $M=\bar S^{-1}DS$ in which $S$ is nonsingular and $D$ is diagonal.
\end{abstract}


 \maketitle

\section{Introduction}

An $n\times n$ complex matrix $A$ is called \emph{coninvolutory} if $\bar AA=I_n$ and \emph{skew-coninvolutory} if $\bar AA=-I_n$ (and so $n$ is even since $\det(\bar AA)\ge 0$). We prove that each matrix of size $n\times n$ with $n\ge 2$  is a sum of 5 coninvolutory matrices and each matrix of size $2m\times 2m$ is a sum of 5 skew-coninvolutory matrices.

These results are somewhat unexpected since the set of matrices that are sums of involutory matrices is very restricted. Indeed, if $A^2=I_n$ and $J$ is the Jordan form of $A$, then $J^2=I_n$,  $J=\diag(1,\dots,1,-1,\dots,-1)$, and so $\trace(A)=\trace(J)$ is an integer.
Thus, if a matrix is a sum of involutory matrices, then its trace is an integer. Wu  \cite[Corollary 3]{wu} and Spiegel \cite[Theorem 5]{Spiegel} prove that an $n\times n$ matrix can be decomposed into a sum of involutory matrices if and only if its trace is an integer being even if $n$ is even.

We also prove that each square complex matrix is a sum of a coninvolutory matrix and a condiagonalizable matrix. A matrix is \emph{condiagonalizable} if it can be written in the form $\bar S^{-1}DS$ in which $S$ is nonsingular and $D$ is diagonal; the set of condiagonalizable matrices is described in \cite[Theorem 4.6.11]{H-J}.

Similar problems are discussed in Wu's survey \cite{wu1}. Wu \cite{wu1} shows that each matrix is a sum of unitary matrices and discusses the number of summands (see also \cite{mer}).
Wu \cite{wu} establishes that $M$ is a sum of idempotent matrices if and only if $\trace(M)$ is an integer and $\trace(M)\ge \rank(M)$.
Rabanovich \cite{rab} proves that every square complex matrix is a linear combination of three idempotent matrices.  Abara, Merino, and Paras \cite{aba} study
coninvolutory and skew-coninvolutory matrices.

\section{Each matrix is a sum of a coninvolutory matrix and a condiagonalizable matrix}

Two matrices $A$ and $B$ over a field $\F$ are \emph{similar} (or, more accurately, \emph{$\F$-similar}) if there exists a nonsingular matrix $S$ over $\F$ such that $S^{-1}AS=B$. A matrix $A$ is \emph{diagonalizable} if it is similar to a diagonal matrix.
Two complex matrices $A$ and $B$ are \emph{consimilar} if there exists a nonsingular matrix $S$ such that $\bar S^{-1}AS=B$; a canonical form under consimilarity is given in \cite[Theorem 4.6.12]{H-J}.
A complex matrix $A$ is \emph{real-condiagonalizable} if it is consimilar to a diagonal real matrix.

By the statement (b) of the following theorem, each square complex matrix is a sum of two condiagonalizable matrices, one of which may be taken to be coninvolutory.

\begin{theorem}\label{t1}
\begin{itemize}
  \item[\rm(a)]
Each square matrix over an infinite field is a sum of an involutory matrix and a diagonalizable matrix.

  \item[\rm(b)]
Each square complex matrix is a sum of a coninvolutory matrix and a real-condiagonalizable matrix.

  \item[\rm(c)]
Each square complex matrix is consimilar to $I_n+D$, in which $D$ is a real-condiagonalizable matrix.

  \item[\rm(d)]
Each square complex matrix is consimilar to $C+D$, in which $C$ is coninvolutory and $D$ is a diagonal real matrix.
\end{itemize}
\end{theorem}

\begin{proof}
The theorem is trivial for $1\times 1$ matrices.

Let $\F$ be any field.
The \emph{companion matrix of a polynomial}
\[
f(x)=x^m-a_1x^{m-1}-\dots-a_m\in\F[x]
\]
is the matrix
\begin{equation}\label{kjw}
F(f):=
\begin{bmatrix}
0&&0&a_m\\1&\ddots&&\vdots\\
&\ddots&0&a_{2}\\0&&1&a_1
\end{bmatrix}\in\F^{m\times m};
\end{equation}
its characteristic polynomial is $f(x)$. By
\cite[Section 12.5]{van},
 \begin{equation}\label{mlr1}
\parbox[c]{0.7\textwidth}{each $A\in\F^{n\times n}$ is $\F$-similar to a direct sum of companion matrices whose characteristic polynomials are powers of prime polynomials; this direct sum is uniquely determined by $A$, up to permutations of summands.}
\end{equation}
 Moreover,
 \begin{equation}\label{mlr}
\parbox[c]{0.7\textwidth}{if $f,g\in\F[x]$ are relatively prime, \\then $F(f)\oplus F(g)$ is $\F$-similar to $F(fg)$.}
\end{equation}
\medskip

(a) Let  $A$ be a matrix of size $n\times n$ with $n\ge 1$
over an infinite field $\F$. It is similar to a direct sum of companion matrices:
\[
SAS^{-1}=B=F_1\oplus\dots\oplus F_t,\qquad S\text{ is nonsingular}.
\]
If $B=C+D$ is the sum of an involutory matrix $C$ and a diagonalizable matrix $D$, then $A=S^{-1}CS+S^{-1}DS$ is also the sum of an involutory matrix and a diagonalizable matrix. Thus, it suffices to prove the statement (a) for $B$. Moreover, it suffices to prove it for an arbitrary companion matrix \eqref{kjw}.

Each matrix
\[
G=\begin{bmatrix}
1&&0&b_m\\&\ddots&&\vdots\\
&&1&b_{2}\\0&&&-1
\end{bmatrix}\in\F^{m\times m}
\]
is involutory. Changing $b_2,\dots,b_{m}$, we get
\[
F(f)-G+I_m=
\begin{bmatrix}
0&&0&c_m\\1&\ddots&&\vdots\\
&\ddots&0&c_{2}\\0&&1&a_1+2
\end{bmatrix}
\]
with arbitrary $c_2,\dots,c_{m}\in\mathbb F$. For each pairwise unequal $\lambda _1,\dots,\lambda _m\in\mathbb F$ such that $\lambda _1+\dots+\lambda _m=a_{1}+2=\text{trace}(F(f)-G+I_m)$, we can take $G$ such that the characteristic polynomial of $F(f)-G+I_m$ is equal to
\[
x^m-(a_1+2)x^{m-1}-c_{2}x^{m-2}-\dots-c_m  =(x-\lambda_1)\cdots(x-\lambda_m).
\]
Thus,
\begin{equation}\label{mry}
\parbox[c]{0.7\textwidth}{$F(f)-G+I_m$ is $\F$-similar to $\diag(\lambda _1,\dots,\lambda _m)$,}
\end{equation}
and so the matrix $F(f)-G$ is diagonalizable.
\medskip

(b)
Let us prove the statement (b) for $A\in \C^{n\times n}$ with $n>1$. By \cite[Corollary 4.6.15]{H-J},
\begin{equation}\label{jdr}
\text{each square complex matrix is consimilar to a real matrix,}
\end{equation}
hence $A=\bar S^{-1}BS$ for some $B\in\R^{n\times n}$ and nonsingular $S\in\C^{n\times n}$. By the statement (a),  $B=C+D$, in which $C\in\R^{n\times n}$ is involutory and $D\in\R^{n\times n}$ is real-diagonalizable. Then $D=R^{-1}ER$, in which $R\in\R^{n\times n}$ is nonsingular and $E\in\R^{n\times n}$ is diagonal. Thus,
$A=\bar S^{-1}CS+\overline{(RS)}^{-1}E(RS)$ is a sum of a coninvolutory matrix and a real-condiagonalizable matrix.

(c) Let $A\in \C^{n\times n}$ with $n>1$. By (b), $A=C+D$, in which $C$ is coninvolutory and $D$ is real-condiagonalizable. By  \cite[Lemma 4.6.9]{H-J}, $C$ is coninvolutory if and only if there exists a nonsingular $S$ such that $C=\bar S^{-1}S$ (that is, $C$ is consimilar to the identity). Then $\bar SAS^{-1}=I_n+\bar SDS^{-1}$, in which $\bar SDS^{-1}$ is real-condiagonalizable.

(d) This statement follows from (b).
\end{proof}

\begin{corollary}\label{rem1}
Each $m\times m$  companion matrix \eqref{kjw} with $m\ge 2$ is $\F$-similar to $G+\diag(\mu _1,\dots,\mu_m)$, in which $G$ is involutory and $\mu _1,\dots,\mu_m\in\F$ are arbitrary pairwise unequal numbers such that $\mu _1+\dots+\mu_m= a_1+2-m$.
\end{corollary}
We get this corollary from \eqref{mry} by taking $\diag(\mu_1,\dots,\mu_m  ):=\diag(\lambda _1,\dots,\lambda _m)-I$.

\section{Each $n\times n$ matrix with $n>1$ is a sum of 5 coninvolutory matrices}

\begin{theorem}\label{t2}
Each $n\times n$ complex matrix with $n\ge 2$ is a sum of 4  coninvolutory matrices if $n=2$ and 5  coninvolutory matrices if $n\ge 2$.
\end{theorem}

\begin{proof}
Let us prove the theorem for  $M\in\C^{n\times n}$. By \eqref{jdr}, $M=\bar S^{-1}AS$ for some $A\in\R^{n\times n}$ and a nonsingular $S$. If $A=C_1+\dots+C_k$ is a sum of coninvolutory matrices, then $M=\bar S^{-1}C_1S+\dots+\bar S^{-1}C_kS$ is also a sum of coninvolutory matrices.

Thus, it suffices to prove  Theorem \ref{t2} for $A\in\R^{n\times n}$.
\medskip

\emph{Case 1: $n=2$.} By \cite[Theorem 3.4.1.5]{H-J},
each  $2\times 2$ real matrix is $\R$-similar to one of the matrices
\begin{equation}\label{ndj}
\begin{bmatrix}
  a & 0 \\
  0 & b \\
\end{bmatrix},\quad
\begin{bmatrix}
  a & 1 \\
  0 & a \\
\end{bmatrix},\quad
\begin{bmatrix}
  a & b \\
  -b & a \\
\end{bmatrix}\ (b>0),\qquad a,b\in\R.
\end{equation}

(i) The first matrix is a sum of 4 coninvolutory matrices since it is represented in the form
\[
\begin{bmatrix}
       a & 0 \\
       0 & b \\
     \end{bmatrix}
=\begin{bmatrix}
       (a-b)/2 & 0 \\
       0 & -(a-b)/2 \\
     \end{bmatrix}+\begin{bmatrix}
       (a+b)/2 & 0 \\
       0 & (a+b)/2 \\
     \end{bmatrix}
\]
and each summand is a sum of two coninvolutory matrices because
\[
\begin{bmatrix}
       2c & 0 \\
       0 & -2c \\
     \end{bmatrix}
=\begin{bmatrix}
       c & 1 \\
       (1-c^2) & -c \\
     \end{bmatrix}+\begin{bmatrix}
       c & -1 \\
       -(1-c^2) & -c \\
     \end{bmatrix}
\] and
\begin{equation}\label{ssd2}
\begin{bmatrix}
       2c & 0 \\
       0 & 2c \\
     \end{bmatrix}
=\begin{bmatrix}
       c & i \\
       (1-c^2)i & c \\
     \end{bmatrix}+\begin{bmatrix}
       c & -i \\
       -(1-c^2)i & c \\
     \end{bmatrix}
\end{equation}
are sums of two coninvolutory matrices for all $c\in\R$.

(ii) The second matrix is a sum of 4 coninvolutory matrices since
\[
\begin{bmatrix}
       a & 1 \\
       0 & a \\
     \end{bmatrix}
=\begin{bmatrix}
       a & 0 \\
       0 & a \\
     \end{bmatrix}+\begin{bmatrix}
       0 & 1 \\
       0 & 0 \\
     \end{bmatrix}
\]
and each summand is a sum of two coninvolutory matrices: the first due to \eqref{ssd2} and the second due to
\[
\begin{bmatrix}
       0 & 1 \\
       0 & 0 \\
     \end{bmatrix}
=\begin{bmatrix}
       1 & 1 \\
       0 & -1 \\
     \end{bmatrix}+\begin{bmatrix}
       -1 & 0 \\
       0 & 1 \\
     \end{bmatrix}.
\]

(iii)
The third matrix is a sum of 4 coninvolutory matrices since
\[
\begin{bmatrix}
       a & b \\
       -b & a \\
     \end{bmatrix}
=\begin{bmatrix}
       a & 0 \\
       0 & a \\
     \end{bmatrix}+\begin{bmatrix}
       0 & b \\
       -b & 0 \\
     \end{bmatrix}
\]
and each summand is a sum of two coninvolutory matrices due to \eqref{ssd2} and
\[
\begin{bmatrix}
       0 & b \\
       -b & 0 \\
     \end{bmatrix}
=\begin{bmatrix}
       1 & b \\
       0 & -1 \\
     \end{bmatrix}+\begin{bmatrix}
       -1 & 0 \\
       -b & 1 \\
     \end{bmatrix}.
\]
Thus, each $2\times 2$ matrix $A$ is a sum of 4 coninvolutory matrices. Applying this statement to $A-I_2$, we get that $A=I_2+(A-I_2)$ is also a sum of 5 coninvolutory matrices.
\medskip

\emph{Case 2: $n$ is even}.
By Theorem \ref{t1}(d),  $A$ is consimilar to $C+D$, where $C$ is coninvolutory and $D$ is a diagonal real matrix, which proves Theorem \ref{t2} in this case due to Case 1 since $D$ is a direct sum of $2\times 2$ matrices.
\medskip

\emph{Case 3: $n$ is odd}.
By \eqref{mlr1}, $A$ is $\R$-similar to a direct sum
\begin{equation}\label{feo}
B=F(f_1)\oplus\dots\oplus F(f_t),\qquad f_i(x)=x^{m_i}-a_{i1}x^{m_i-1}
-\dots-a_{im}\in\R[x].
\end{equation}

We can suppose that $m_1>1$. Indeed, if $m_i>1$ for some $i$, then we interchange $F(f_1)$ and $F(f_i)$. Let $m_1=\dots=m_t=1$ and let $a_{11}
\ne 0$ (if $B=0$, then $B=I+(-I)$ is the sum of involutory matrices). If $a_{11}=a_{21}$, then we replace $a_{11}$ by $-a_{11}$ using the consimilarity of $[a_{11}]$ and $[-a_{11}]$.
By \eqref{mlr}, $F(f_1)\oplus F(f_2)=[a_{11}]\oplus[a_{21}]$ is $\R$-similar to $F((x-a_{11})(x-a_{21}))$.

We obtain $B$ of the form $F(f_1)\oplus C$ with $m_1>1$.
By Corollary  \ref{rem1}, $F(f_1)$ is $\R$-similar to $G+\diag(\mu _1,\dots,\mu_{m_1})$, in which $G$ is a real involutory matrix and $\mu _1,\dots,\mu_{m_1}\in\R$ are arbitrary pairwise unequal numbers such that $\mu _1+\dots+\mu_{m_1}= a_{11}+2-{m_1}$.

We take $\mu _1=2$ (and then $\mu_2=-2$) if $f_1(x)=x^2-a_{12}$. We take  $\mu _1=0$ if $f_1(x)\ne x^2-a_{12}$.
Applying Theorem \ref{t1}(d) to the other direct summands $F(f_2),\dots, F(f_t)$, we find that $B$ is $\R$-similar to \[
\begin{bmatrix}
  G & 0\\
  0& C\\
\end{bmatrix}+
\begin{bmatrix}
  \mu_1 & 0\\
  0& D\\
\end{bmatrix},
\]
in which the first summand is coninvolutory and the second is a diagonal real matrix. By Case 1,
\[
D=C_1+C_2+C_3+C_4,
\]
in which $C_1,C_2,C_3,C_4$ are coninvolutory matrices. Then
\[
\begin{bmatrix}
  \mu_1 & 0 \\
  0 & D \\
\end{bmatrix}=
\begin{bmatrix}
  1 & 0 \\
  0 & C_1 \\
\end{bmatrix}
+
\begin{bmatrix}
  \mu_1-1 & 0 \\
  0 & C_2 \\
\end{bmatrix}
+
\begin{bmatrix}
  1 & 0 \\
  0 & C_3 \\
\end{bmatrix}
+\begin{bmatrix}
  -1 & 0 \\
  0 & C_4 \\
\end{bmatrix}
\]
is a sum of 4 coninvolutory matrices.
\end{proof}

\section{Each $2m\times 2m$ matrix is a sum of 5 skew-coninvolutory matrices}

We recall that an $n\times n$ complex matrix $A$ is called \emph{skew-coninvolutory} if $\bar AA=-I_n$ (and so $n$ is even since $\det(\bar AA)\ge 0$).

\begin{theorem}\label{ts}
Each $2m\times 2m$ complex matrix is a sum of at most 5  skew-coninvolutory matrices.
\end{theorem}

\begin{proof}
Let us prove the theorem for  $A\in\C^{2m\times 2m}$. If $A=\bar S^{-1}BS$ and $B=C_1+\dots+C_k$ is a sum of skew-coninvolutory matrices, then $A=\bar S^{-1}C_1S+\dots+\bar S^{-1}C_kS$ is a sum of skew-coninvolutory matrices too.
Thus, it suffices to prove the theorem for any matrix that is consimilar to $A$.

By \cite[Theorem 4.6.12]{H-J}, each square complex matrix is consimilar to a direct sum, uniquely determined up to permutation of summands, of matrices of the following two types:
\begin{equation}\label{mmu}
J_n(\lambda ):=\begin{bmatrix}
  \lambda  & 1&&0 \\
&\lambda  & \ddots \\
&&\ddots&1\\
0&&&\lambda   \\
\end{bmatrix}\quad (n\text{-by-}n,\ \lambda \in\R,\ \lambda \ge 0)
\end{equation}
and
\begin{equation}\label{ccd}
H_{2m}(\mu):=\begin{bmatrix}
  0& I_n \\
J_n(\mu )&0\\
\end{bmatrix}\quad (\mu \in\C,\ \mu<0\text{ if }\mu \in\R).
\end{equation}
Thus, we suppose that $A$ is a direct sum of matrices of these types.
\medskip

\emph{Case 1: $A$ is diagonal.} Then $A$ is a sum of 4 skew-coninvolutory matrices since $A$ is a direct sum of $m$ real diagonal 2-by-2 matrices and each  real diagonal 2-by-2 matrix is represented in the form
\[
\begin{bmatrix}
       a & 0 \\
       0 & b \\
     \end{bmatrix}
=\begin{bmatrix}
       (a-b)/2 & 0 \\
       0 & -(a-b)/2 \\
     \end{bmatrix}+\begin{bmatrix}
       (a+b)/2 & 0 \\
       0 & (a+b)/2 \\
     \end{bmatrix}
\]
in which each summand is a sum of two skew-coninvolutory matrices because
\[
\begin{bmatrix}
       2c & 0 \\
       0 & -2c \\
     \end{bmatrix}
=\begin{bmatrix}
       c & -1 \\
       (1+c^2) & -c \\
     \end{bmatrix}+\begin{bmatrix}
       c & 1 \\
       -(1+c^2) & -c \\
     \end{bmatrix}
\] and
\begin{equation}\label{bfl}
\begin{bmatrix}
       2c & 0 \\
       0 & 2c \\
     \end{bmatrix}
=\begin{bmatrix}
       c & -i \\
       (1+c^2)i & c \\
     \end{bmatrix}+\begin{bmatrix}
       c & i \\
       -(1+c^2)i & c \\
     \end{bmatrix}
\end{equation}
are sums of two skew-coninvolutory matrices  for all $c\in\R$.
\medskip

\emph{Case 2: $A$ is a direct sum of matrices of type \eqref{mmu}.} Then it has the form
\begin{equation*}\label{saj}
A=\begin{bmatrix}
  \lambda_1  & \varepsilon _1&&0 \\
&\lambda_2  & \ddots \\
&&\ddots&\varepsilon _{2m-1}\\
0&&&\lambda_{2m}   \\
\end{bmatrix}
\end{equation*}
in which all $\lambda _i\ge 0$ and all $\varepsilon_i\in\{0,1\}$.

Represent $A$ in the form $A=C+D$, in which
\[
C:=\begin{bmatrix}
     c_1 & 1 \\
     -1+c_1^2 & -c_1 \\
   \end{bmatrix}\oplus\dots\oplus
 \begin{bmatrix}
     c_m & 1 \\
     -1+c_m^2 &- c_m \\
   \end{bmatrix},\quad \text{all }c_i\in\R,
\]
is a skew-coninvolutory matrix. Let us show that $c_1,\dots,c_m$ can be chosen such that all eigenvalues of $D$ are distinct real numbers.

The matrix $D$ is upper block-triangular with the diagonal blocks
\[
D_1:=\begin{bmatrix}
     \lambda _1-c_1 & \varepsilon _1-1 \\
     1-c_1^2 & \lambda _2+c_1 \\
   \end{bmatrix},\ \dots, \
D_m:= \begin{bmatrix}
     \lambda _{2m-1}-c_m &\varepsilon _{2m-1}- 1 \\
     1-c_m^2 &\lambda _{2m}+ c_m \\
   \end{bmatrix}.
\]
Hence, the the set of eigenvalues of $D$ is the union of the sets of eigenvalues of $D_1,\dots ,D_m$.

Let $c_1,\dots,c_{k-1}$ have been chosen such that the eigenvalues of $D_1,\dots,D_{k-1}$ are distinct real numbers $\nu _1,\dots,\nu _{2k-2}$. Depending on $\varepsilon _{2k-1}\in\{0,1\}$, the matrix $D_k$ is
\begin{equation}\label{gfe}
\begin{bmatrix}
     \lambda _{2k-1}-c_k & -1 \\
     1-c_k^2 & \lambda _{2k}+c_k \\
   \end{bmatrix}
   \quad\text{or}\quad
\begin{bmatrix}
     \lambda _{2k}-c_k & 0 \\
     1-c_k^2 & \lambda _{2k}+c_k \\
   \end{bmatrix}.
\end{equation}

\begin{itemize}
  \item
Let $D_k$ be the first matrix in \eqref{gfe}. Its characteristic polynomial is
\begin{align*}
 \chi _k(x)&=x^2-\trace(D_k)x+\det(D_k)\\&=
x^2-(\lambda _{2k-1}+\lambda _{2k})x+(  \lambda _{2k-1}-c_k)(\lambda _{2k}+c_k)+1-c_k^2.
\end{align*}
Its discriminant is
\begin{align*}
\Delta _k=&(\lambda _{2k-1}+\lambda _{2k})^2-4[\lambda _{2k-1}\lambda _{2k}+(\lambda _{2k-1}-\lambda _{2k})c_k-2c_k^2+1]\\
=&(\lambda _{2k-1}-\lambda _{2k})^2+4 (-\lambda _{2k-1}+\lambda _{2k})c_k+8c_k^2-4.
\end{align*}
For a sufficiently large $c_k$, $\Delta _k>0$ and so the roots of $\chi _k(x)$ are some distinct real numbers $\nu _{2k-1}$ and $\nu _{2k}$. Since \[\nu _{2k-1}+\nu _{2k}=\trace(D_k)=\lambda _{2k-1}+\lambda _{2k},\] we have
\begin{align*}
\det(D_k)&=\nu _{2k-1}\nu _{2k}=\nu _{2k-1}
(\lambda _{2k-1}+\lambda _{2k}-\nu _{2k-1})\\&=
(\lambda _{2k-1}+\lambda _{2k}-\nu _{2k})\nu _{2k}.
\end{align*}
Taking $c_k$ such that
\[
\det(D_k)\ne \nu _{i}
(\lambda _{2k-1}+\lambda _{2k}-\nu _{i})\quad\text{for all }i=1,\dots,2k-2,
\]
we get $\nu _{2k-1}$ and $\nu _{2k}$ that are not equal to $\nu _{1},\dots,\nu _{2k-2}$.

  \item
Let $D_k$ be the second matrix in \eqref{gfe}. Then its eigenvalues are  $\lambda _{2k}-c_k$ and $\lambda _{2k}+c_k$. We choose a nonzero real $c_k$ such that these eigenvalues are not equal to $\nu _{1},\dots,\nu _{2k-2}$.
\end{itemize}
We have constructed the real skew-coninvolutory matrix $C$ such that $A=C+D$, in which $D$ is a real matrix with distinct eigenvalues $\nu _{1},\dots,\nu _{2m}\in \R$. Since $D$ is $\R$-similar to a diagonal matrix and by Case 1, $D$ is a sum of 4 skew-coninvolutory matrices.
\medskip

\emph{Case 3: $A$ is a direct sum of matrices of types \eqref{mmu} and \eqref{ccd}.} Due to Case 2, it suffices to prove that each matrix $H_{2m}(\mu)$ is a sum of 5 skew-coninvolutory matrices. Write
\[
\begin{bmatrix}
  0& I_n \\
J_n(\mu )&0\\
\end{bmatrix}=\begin{bmatrix}
  0& I_n \\
-I_n&0\\
\end{bmatrix}+\begin{bmatrix}
  0& 0 \\
J_n(\mu )+I_n&0\\
\end{bmatrix}.
\]
The first summand is a skew-coninvolutory matrix, and so we need to proof that the second summand is a sum of 4 skew-coninvolutory matrices. By \eqref{jdr}, there exists a nonsingular $S$ such that $B:=\bar S^{-1}(J_n(\mu )+I_n)S$ is a real matrix. Then
the second summand is consimilar to a real matrix:
\[
\begin{bmatrix}
  \bar S^{-1}&0 \\
0&\bar S^{-1}\\
\end{bmatrix}
\begin{bmatrix}
  0& 0 \\
J_n(\mu )+I_n&0\\
\end{bmatrix}
\begin{bmatrix}
  S&0 \\
0&S\\
\end{bmatrix}=
\begin{bmatrix}
  0&0 \\
B&0\\
\end{bmatrix},
\]
which is the sum of two coninvolutory matrices:
\begin{equation}\label{nur}
\begin{bmatrix}
  0&0 \\
B&0\\
\end{bmatrix}=
\begin{bmatrix}
  I_n&0 \\
B&-I_n\\
\end{bmatrix}+
\begin{bmatrix}
  -I_n&0 \\
0&I_n\\
\end{bmatrix}.
\end{equation}

By  \cite[Lemma 4.6.9]{H-J}, each coninvolutory matrix is consimilar to the identity matrix. Hence, each summand in \eqref{nur} is consimilar to $I_{2n}$, which is a sum of two skew-coninvolutory matrices due to \eqref{bfl}. Thus, the matrix \eqref{nur} is a sum of 4 skew-coninvolutory matrices.
\end{proof}

\section*{Acknowledgments}
The work of V.V. Sergeichuk was done during his visit to the University of S\~ao Paulo supported by FAPESP, grant 2015/05864-9.

\end{document}